\documentclass[12pt, letterpaper]{amsart}

\usepackage{amssymb}
\usepackage{mathrsfs}
	\normalbaselines						  %

\newcommand{\R}{{\mathbb  R}}
\newcommand{\kf}[2]{{\begin{pmatrix} #1 \\ #2 \end{pmatrix}}}
\newcommand{\kg}{{\begin{pmatrix} f \\ g \end{pmatrix}}}
\newcommand{\D}{{\mathbb  D}}

\newcommand{\T}{\mathbb{T}}

\newcommand{\Z}{{\mathbb  Z}}
\newcommand{\N}{{\mathbb  N}}
\newcommand{\C}{{\mathbb  C}}

\newcommand{\OZ}{{\mathbf{O}}}

\newcommand{\ID}{{\mathbf{1}}}

\newcommand{\OID}{{\mathbf{I}}}

\newcommand{\fdot}{\,\cdot\,}

\newcommand{\calC}{{\mathcal C}}
\newcommand{\cH}{\mathcal{H}}
\newcommand{\cD}{\mathcal{D}}

\newcommand{\te}{\theta}
\newcommand{\f}{\varphi}

\newcommand{\e}{\varepsilon}

\newcommand{\Kft}{{\begin{pmatrix}H^2(\mathcal{E}_\ast)\\\clos\bigtriangleup L^2(\mathcal{E})\end{pmatrix}
\ominus \begin{pmatrix}\Theta\\\bigtriangleup\end{pmatrix} H^2(\mathcal{E})}}
\newcommand{\kr}{{\begin{pmatrix} \te \\ \bigtriangleup \end{pmatrix}}}

\DeclareMathOperator{\clos}{clos}

\DeclareMathOperator{\supp}{supp}

\newcommand{\ci}[1]{_{ {}_{\scriptstyle #1}}}
\newcommand{\ti}[1]{_{\scriptstyle \text{\rm #1}}}
\catcode`@=11
\chardef\mathlig@atcode\count255

\def\actively#1#2{\begingroup\uccode`\~=`#2\relax\uppercase{\endgroup#1~}}
\def\mathlig@gobble{\afterassignment\mathlig@next@cmd\let\mathlig@next= }
\def\mathlig@delim{\mathlig@delim}
\def\mathlig@defcs#1{\expandafter\def\csname#1\endcsname}
\def\mathlig@let@cs#1#2{\expandafter\let\expandafter#1\csname#2\endcsname}
\def\mathlig@appendcs#1#2{\expandafter\edef\csname#1\endcsname{\csname#1\endcsname#2}}

\def\mathlig#1#2{\mathlig@checklig#1\mathlig@end\mathlig@defcs{mathlig@back@#1}{#2}\ignorespaces}

\def\mathlig@checklig#1#2\mathlig@end{%
 \expandafter\ifx\csname mathlig@forw@#1\endcsname\relax
 \expandafter\mathchardef\csname mathlig@back@#1\endcsname=\mathcode`#1%
 \mathcode`#1"8000\actively\def#1{\csname mathlig@look@#1\endcsname}%
 \mathlig@dolig#1\mathlig@delim
\fi
\mathlig@checksuffix#1#2\mathlig@end
}

\def\mathlig@checksuffix#1#2\mathlig@end{%
\ifx\mathlig@delim#2\mathlig@delim\relax\else\mathlig@checksuffix@{#1}#2\mathlig@end\fi
}
\def\mathlig@checksuffix@#1#2#3\mathlig@end{%
\expandafter\ifx\csname mathlig@forw@#1#2\endcsname\relax\mathlig@dosuffix{#1}{#2}\fi
\mathlig@checksuffix{#1#2}#3\mathlig@end
}

\def\mathlig@dosuffix#1#2{%
\mathlig@appendcs{mathlig@toks@#1}{#2}%
\mathlig@dolig{#1}{#2}\mathlig@delim
}

\def\mathlig@dolig#1#2\mathlig@delim{%
 \mathlig@defcs{mathlig@look@#1#2}{%
 \mathlig@let@cs\mathlig@next{mathlig@forw@#1#2}\futurelet\mathlig@next@tok\mathlig@next}%
 \mathlig@defcs{mathlig@forw@#1#2}{%
  \mathlig@let@cs\mathlig@next{mathlig@back@#1#2}%
  \mathlig@let@cs\checker{mathlig@chck@#1#2}%
  \mathlig@let@cs\mathligtoks{mathlig@toks@#1#2}%
  \expandafter\ifx\expandafter\mathlig@delim\mathligtoks\mathlig@delim\relax\else
  \expandafter\checker\mathligtoks\mathlig@delim\fi
  \mathlig@next
 }%
 \mathlig@defcs{mathlig@toks@#1#2}{}%
 \mathlig@defcs{mathlig@chck@#1#2}##1##2\mathlig@delim{%
  \ifx\mathlig@next@tok##1%
   \mathlig@let@cs\mathlig@next@cmd{mathlig@look@#1#2##1}\let\mathlig@next\mathlig@gobble
  \fi
  \ifx\mathlig@delim##2\mathlig@delim\relax\else
   \csname mathlig@chck@#1#2\endcsname##2\mathlig@delim
  \fi
 }%
 \ifx\mathlig@delim#2\mathlig@delim\else
  \mathlig@defcs{mathlig@back@#1#2}{\csname mathlig@back@#1\endcsname #2}%
 \fi
}%

\catcode`@\mathlig@atcode

\mathchardef\ordinarycolon\mathcode`\:
\def\vcentcolon{\mathrel{\mathop\ordinarycolon}}
\mathlig{:=}{\vcentcolon=}
\mathlig{::=}{\vcentcolon\vcentcolon=}
\numberwithin{equation}{section}

\theoremstyle{plain}
\newtheorem{theo}{Theorem}[section]

\newtheorem{lem}[theo]{Lemma}

\theoremstyle{definition}

\theoremstyle{remark}
\newtheorem*{ex*}{Example}
\theoremstyle{remark}
\newtheorem*{exs*}{Examples}
\theoremstyle{remark}
\newtheorem*{rem}{Remark}
\theoremstyle{remark}
\newtheorem*{rems}{Remarks}

\title[Rank one and finite rank perturbations]{Rank one and finite rank perturbations}
\author{Constanze~Liaw}
\thanks{The author is supported by the NSF grant DMS-1101477.}
\address{Department of Mathematics, Texas A\&M University, Mailstop 3368,      
 College Station, TX  77843, USA}
\email{conni@math.tamu.edu}
\urladdr{http://www.math.tamu.edu/$\sim$conni}


\begin{document}

\begin{abstract}
We survey the relationships of rank one self-adjoint and unitary perturbations as well as finite rank unitary perturbations with various branches of analysis and mathematical physics. We include the case of non-inner characteristic operator functions. For rank one perturbations and non-inner characteristic functions, we prove a representation formula for the adjoint of the Clark operator. Throughout we mention many open problems at varying levels of difficulty.
\end{abstract}

\maketitle

%%%%%%%%%%%%%%%%%%%%%%%%%%%%%
\section{Introduction}
%%%%%%%%%%%%%%%%%%%%%%%%%%%%%%%%%
In general, perturbation theory is concerned with the following question: Given an operator $A$ what can we say about the properties of an operator $A+B$? The statements vary, depending on which operator class $B$ is taken from.
In this survey we take the perturbation $B$ from very restricted operator classes: Rank one or finite rank operators. Although all of classical perturbation theory (e.g.~Weyl--von Neumann and Kato--Rosenblum) can be applied, rank one perturbations gave rise to a surprisingly rich theory with the major difficulty being the instability of the singular part.% Some partial answers are summarized below.

Self-adjoint rank one perturbations occurred naturally in mathematical physics \cite{weyl}. The study of rank one unitary perturbations was initiated by Clark \cite{Clark}, and many deep results were proved by Aleksandrov (see \cite{poltsara2006} and the references within). Their interest in the subject likely originated from the invariant subspace problem.

Within analysis many fruitful connections with holomorphic composition operators, rigid functions and the Nehari interpolation problem have been discovered and studied, see for example, \cite{poltsara2006}. Further the problem is connected to the invariant subspace problem, see for example, page 192ff of \cite{cimaross}. Recently, Jury \cite{Jury} computed the asymptotic symbols of a certain class of weakly asymptotic Toeplitz operators in terms of the Aleksandrov--Clark measures which occur in the context of rank one perturbations.

In mathematical physics, rank one and finite rank perturbations are of interest for several reasons. First of all, rank one perturbations are used in some results on orthogonal polynomials, and they are related to free probability.
Further, a change in the boundary condition of a limit-point half-line Schr\"odinger operator from Dirichlet to Neumann, or to mixed conditions, can be reformulated as adding a rank one perturbation (see for example \cite{SIMREV}). They, moreover, have connections with Anderson models.
They are related to the famous conjecture about Anderson localization \cite{And1958}. While many specializations of this conjecture were studied in literature and the field is still very active (see e.g.~\cite{AizMol1993, denseGdelta, FrSp, Germ, Klo2}, also see \cite{Kir95, Ban09} for a recent account of parts of the field), one can roughly summarize the {\bf conjecture} as follows: Arbitrarily small random impurities in a crystal will cause spatial localization of electrons.

Rank one unitary perturbations were generalized to higher rank unitary perturbations (for example, in \cite{Ball1978, AlexKapustin}), which describe the situation when the perturbed operator is not cyclic. While rank one perturbations occur when considering the change in boundary condition for the discrete Schr\"odinger operator on $l^2(\N)$, a mathematical physics interpretation of rank $n$ perturbations can be interpreted as a change in boundary conditions for a discrete Schr\"odinger operator on  $l^2(\N^n)$.

It is worth mentioning that thinking of a rank two perturbation as a rank one perturbation of a rank one perturbation is usually not insightful beyond the realms of classical perturbation theory, the reason being that we do not know enough about the spectral measure of the first rank one perturbation.

Here finite rank perturbations are investigated in three settings: (a) In formulating the problem on an abstract Hilbert space, we can apply classical perturbation theory. (b) The subtleties of this very restrictive perturbation problem can be captured conveniently in its spectral representation, making the spectral measure the main object of interest. (c) The third perspective to finite rank unitary perturbations involves model and dilation theory, providing a rather geometric point of view. We mostly restrict the model theoretic part of this survey to the model theory of Sz.-Nagy and Foia\c s.

In Section \ref{s-SA} we introduce self-adjoint rank one perturbations and describe some preliminary results often referred to as Aronszajn--Donoghue theory. In Section \ref{s-EXA} we provide a glimpse at some of the many examples of ``mixed" spectral measures. We then (Section \ref{s-AndHam}) briefly explain the Simon--Wolff criterion and other connections with Anderson models.
In Section \ref{s-GenMod} we introduce the relationship between Sz.-Nagy and Foia\c s model theory and finite rank unitary perturbations of a completely non-unitary contraction.
A generalization of Clark's construction to the case of non-inner characteristic functions is described in Section \ref{s-ACmeas}. Further, we introduce the normalized Cauchy transform, a key object of Aleksandrov--Clark theory, and explain what is known about its role in the case of this more general construction. A representation of the adjoint of the Clark operator in the case of non-inner characteristic function is given in Theorem \ref{t-repr} (and proved in the appendix, Section \ref{s-appendix}).
The focus of Section \ref{s-pseudo} lies on pseudocontinuation. We describe a way to generalize the result of Douglas--Shapiro--Shields to non-inner characteristic functions \cite{Kapustin2} and the problems that arise when considering operator-valued characteristic functions \cite{AlexKapustin}.
In Section \ref{s-last} we mention some other model theoretic approaches to rank one perturbations.

This survey does not include all the material related to the theory of rank one perturbations. For example, many interesting connections were covered in \cite{cimaross, poltsara2006, SIMREV}.

{\bf Acknowledgement.}
I would like to thank R.~G.~Douglas for the many fruitful discussions about the model theory and for proof-reading parts of this paper.

%%%%%%%%%%%%%%%%%%%%%%%%%%%%%
\section{Self-adjoint rank one perturbations}\label{s-SA}
%%%%%%%%%%%%%%%%%%%%%%%%%%%%%%%%%%%%%%%%%%%%%%%
Let $A$ be a self-adjoint (possibly unbounded) cyclic operator on a Hilbert space $\mathcal H$, and let $\f$ be a cyclic vector for $A$. Consider the family of rank-one perturbations formally given by
\begin{align}\label{d-SA}
A_\alpha = A+\alpha(\,\cdot\,,\f)\f.
\end{align}

\begin{rem}Here, if the operator $A$ is bounded, then $\f$ is a vector in $\mathcal H$. For unbounded $A$, we consider the wider class of ``singular form-bounded" perturbations where we assume $\varphi\in\mathcal H_{-1}(A)\supset \mathcal H$, where $\mathcal H_{-r}(A):=\clos{\cH}$, $r>0$, with respect to the norm $\| \cdot\|\ci{\cH_{-r}(A)} = \| (I +|A|)^{-r/2}\cdot\|\ci{\cH}$. In particular, the perturbation $\alpha (\fdot, \f)\f$ can be unbounded (see \cite{KL, mypaper} and the references within for further details). For $r\ge2$, however, the formal expression \eqref{d-SA} does not uniquely determined a self-adjoint operator: For fixed $\alpha$ there is a family of self-adjoint operators corresponding to \eqref{d-SA}.
\end{rem}

The vector $\f$ is cyclic for $A_\alpha$, since it is cyclic for $A$. Let $\mu_\alpha$ denote the spectral measure of the operator $A_\alpha$ with respect to $\f$.

The Weyl--von Neumann Theorem asserts the invariance of the essential spectrum under compact perturbations (see, for example, \cite{katobook}). Therefore, we have $\sigma\ti{ess}(A_\alpha) = \sigma\ti{ess}(A_\beta)$. Further, the Kato--Rosenblum Theorem states that the absolutely continuous part of the spectral measure remains invariant up to equivalence under trace class perturbations (see, for example, \cite{katobook, SIMREV}). It is worth mentioning that Carey and Pincus \cite{CP} provide a complete characterization in terms of the operators' spectrum of when two operators are unitarily equivalent up to a trace class perturbation.

It was a surprise when Donoghue \cite{Donoghue} discovered an example of a rank one perturbation under which the spectral type was unstable. The subject developed into Aronszajn--Donoghue theory (see e.g.~\cite{SIMREV}). A complete characterization of the absolutely continuous and the pure point parts of the perturbed operators' spectral measure in terms of that of the unperturbed operator. Namely, we have
\[
(d\mu_\alpha)\ti{ac} (x) = \pi^{-1} \lim\limits_{\e\downarrow 0} \left[ \frac{\text{Im} (F(x+i\e))}{|1+\alpha F(x+i\e)|^2}\right],
\]
where $F(x+i\e) = \int_\R \frac{d\mu(t)}{t-(x+i\e)}$ and $\text{Im} (z)$ refers to the imaginary part of $z\in \C$. It is not hard to see that $(\mu_\alpha)\ti{ac}$ and  $(\mu_\beta)\ti{ac}$, $\alpha, \beta \in \R$, are equivalent measures (i.e.~they are mutually absolutely continuous in the sense of Radon--Nikodym). The operator $A_\alpha$ has an eigenvalue at $x$ if and only if the two conditions $\lim\limits_{\e\downarrow 0} F(x+i\e) = -\alpha^{-1}$ and $G(x)<\infty$, $G(x) :=\int_\R\frac{d\mu(t)}{(t-x)^2}$, are satisfied. In fact, we have
\[
(d\mu_\alpha)\ti{pp}(x)  = \sum\limits_{x_n \text{ is eigenvalue}} \frac{\delta(x-x_n)}{\alpha^2 G(x_n)}\,.
\]
Further, it is not hard to see that the singular parts $(\mu_\alpha)\ti{s}$ and  $(\mu_\beta)\ti{s}$, $\alpha, \beta \in \R$ are mutually singular for $\alpha\neq\beta$.

The behavior of the singular continuous spectrum has proved more difficult to capture. {\bf Open problem:} Characterize the singular continuous part of the perturbed operator's spectral measure in terms of the unperturbed operator.

Several sufficient conditions for the absence of singular continuous spectrum are known (see, for example, \cite{cimaross, mypaper, SIMREV}). Based on a relationship between rank one perturbations and the two weight problem for the Hilbert transform (which arises when considering the spectral representation of the perturbed operator $A_\alpha$) one of the sufficient conditions \cite{mypaper} for the absence of the singular spectrum says: Operators $A_\alpha$, $\alpha\neq 0$ have a pure absolutely continuous spectrum on a closed interval $I$, if
\begin{align}\label{e-cond}
\int_0^\e x^{-2} w^\ast_I dx =\infty.
\end{align}
Here $d\mu = w dx+ d\mu\ti{s}$, $w\in L^1(dx)$, and $w^\ast_I$ denotes the increasing rearrangement of $w$ on $I$.
The paper includes a construction of unperturbed operators $A$ with arbitrary embedded singular spectrum and for which all of the perturbed operators $A_\alpha$, $\alpha\neq 0$ have no embedded singular spectrum. This construction is one of many examples.

%%%%%%%%%%%%%%%%%%%%%%%%%%%%%
\section{Instability of the singular spectrum}\label{s-EXA}
%%%%%%%%%%%%%%%%%%%%%%%%%%%%%%%%%%%%%%%%%%%%%%%
Since the example of Donoghue showed that the type of the singular spectrum can change under rank one perturbations, the question was: To which extent may the spectral properties of the measures $\mu_\alpha$ vary as we change $\alpha$. Much work has been done and many interesting examples were discovered, several are included in \cite{SIMREV}. 

First of all notice that the theory of rank one perturbations for pure point and the singular continuous spectrum is quite different in nature. Indeed, it is possible for $A_\alpha$ to have purely singular continuous spectrum on the interval $[0,1]$ for all $\alpha$. But the same behavior is not possible for pure point spectrum. In fact, the perturbations $A_\alpha$ are pure point for all $\alpha$ if and only if the spectrum is countable.

Another question concerns the type of parameter sets that allow dense singular embedded (in absolutely continuous) spectrum. For several years, all examples exhibited dense singular embedded spectrum only for a Lebesgue measure zero set of parameters $\alpha$.
It came as a surprise when Del Rio, Fuentes and Poltoratski \cite{riofuepolt2002} proved the existence of a family of rank one perturbations with dense absolutely continuous spectrum and dense singular spectrum for almost every parameter $\alpha$ in an arbitrary (previously given) set $B\subset\R$ and with purely absolutely continuous spectrum for almost every $\alpha \in\R\backslash B$. Their proof uses Clark theory. Via a complicated construction they show the existence of a characteristic function for which the corresponding family of rank one unitary perturbations has the desired properties. In fact, it possible to produce most any type of singular spectrum in this setting, see \cite{riofuepolt2}. {\bf Open question} (see \cite{riofuepolt2}): Fix an interval $I\subset \R$ and subset $B\subset \R$. Can one find a family of measures $\mu_\beta$ so that $(\mu_\beta)\ti{s}(J)>0$ if and only if $\beta \in B$ and $(\mu_\beta)\ti{ac}(J)>0$ for all $\beta \in \R$ and for every subset $J\subset I$?

Another class of examples is concerned with the question of how unstable the spectral type may be, if we do not have absolutely continuous part. A result of Del Rio, Makarov and Simon \cite{denseGdelta} which was independently proved by Gordon \cite{Gordon1997} states the following. Consider $I\subset\supp\mu$ closed and not a singleton. If $\mu_\alpha|_I$ is singular, then the set of $\alpha$'s for which $\mu_\alpha$ is purely singular continuous is a dense $G_\delta$ set.

A converse to this statement was proved by C.~Sundberg (private communication): For any closed subinterval $I$ which is not a singleton and any $G_\delta$ subset of $\R$, there exists a family of measures corresponding to a family of rank one perturbations such that $\supp \mu\subset I$, $\mu_\alpha$ is purely singular continuous for $\alpha \in G$ and $\mu_\alpha$ is pure point for $\alpha \in \R\backslash G$. In the proof, Sundberg applies Clark theory. He constructs the characteristic function by defining a function on a Riemann surface $\mathcal R$ over the disk $\D$, and then applies the projection from $\mathcal R$ to $\D$.

%%%%%%%%%%%%%%%%%%%%%%%%%%%%%
\section{Self-adjoint rank one perturbations and Anderson-type Hamiltonians}\label{s-AndHam}
%%%%%%%%%%%%%%%%%%%%%%%%%%%%%%%%%%%%%%%%%%%%%%%
Anderson localization concerns an infinitely large crystal with ``small" impurities. These impurities are expressed in terms of an ``on-site potential" and associated with random variables with distribution given by a probability measure $\mathbb{P}$ on $\R^\infty$. The Anderson localization conjecture states that electrons of arbitrary energy will remain within a bounded region in space for all times (almost surely with respect to $\mathbb{P}$). Anderson localization has developed into a rich subject within mathematical physics. Many related operators are investigated and many notions of the term ``localization" were and are being studied. {\bf Vague open problem:} Many conjectures of Anderson localization are unsolved for dimension two and higher.

Let us introduce probably the simplest and most important operator, the discrete random Schr\"odinger operator. Consider the vectors $\delta_n\in l^2(\Z^d)$ with $\delta_j=0$ if $j\neq n$ and 1 if $j = n$ (where $j,n\in \Z^d$). Let $\omega= (\omega_1, \omega_2, \hdots)$ where the $\omega_i$'s are independent, identically distributed random variables with distribution $d p = f dx$, $f\in L^1$. Let $\mathbb{P} = \Pi_n dp$.
The discrete random Schr\"odinger operator is given by
\[
(\bigtriangleup_\omega) u (j)=  - \sum_{|n|=1} (u(j+n) - u(j)) + \omega_j u(j)
\]
consisting of the Laplacian plus the random on-site potential. The discrete Schr\"odinger operator describes the situation where the atoms of the crystal are located at the integer lattice points of $\Z^d$. Adding the random part can be interpreted as having the atoms being not perfectly on the lattice points, but randomly displaced according to the probability distribution $\mathbb{P}$. {\bf More concise open problems:} If $d\ge 2$, does the discrete random Schr\"odinger operator with $\mathbb{P} = \Pi ((2c)^{-1} \chi_{[-c,c]} dx)$ have purely singular spectrum for all $c>0$. This statement is known for $c$ above a certain threshold (which is dependent on the dimension). The same question is open if $\omega_j = \pm1$ with equal probability.

Simon and Wolff \cite{SIMWOL} found a relationship between rank one perturbations and discrete random Schr\"odinger operators. They used the theory of rank one perturbations in order to prove an important case of the, at the time, almost 50-year-old conjecture of Anderson localization in the case of one dimension.
Moreover, the following three results are based on this discovery by Simon and Wolff.

Simon \cite{Sim1994} proved that all $\delta_n$ are cyclic for $\bigtriangleup_\omega$ restricted to an interval $I = [a,b]$ for almost all $\omega\in\mathbb{P}$, if we assume that almost surely (wrt $\omega\in\mathbb{P}$) $I\subset\sigma(\bigtriangleup_\omega)$ and for the pure point spectrum we have $\sigma\ti{pp}(\bigtriangleup_\omega)\cap I = I$ almost surely.

Jaksic and Last \cite{JakLast2000, JakLast2006} extended the ideas of Simon and Wolff to determine most of what is known about the spectral structure of Anderson-type Hamiltonians:
\[
H_\omega = A + \sum_n \omega_n (\fdot, \f_n)\f_n,
\]
where $A$ is a self-adjoint operator on a separable Hilbert space, the $\f_n$ form an orthonormal basis and the $\omega = (\omega_1, \omega_2,\hdots)$ are the random variables.
Anderson-type Hamiltonians are a generalization of the discrete random Schr\"odinger operator.
Jaksic and Last proved, for example, that the singular part $(H_\omega)\ti{s}$ is simple almost surely by showing that certain (complicated) vectors are cyclic.

As refinement of the work by Jaksic and Last it was shown \cite{AbaLiawPolt} that if there exists a vector that is cyclic for $H_\omega$ almost surely, then any non-zero vector is cyclic for $H_\omega$ almost surely. In particular, it follows that any non-zero vector is cyclic for $(H_\omega)\ti{s}$ almost surely.

Another relationship between rank one perturbations and Anderson-type Hamiltonians was discovered, see \cite{Liaw2010}. In some sense and under certain conditions, rank one perturbations are as difficult as Anderson-type Hamiltonians. More precisely, under some assumptions, two realizations $H_\omega$ and $H_\eta$ are almost surely $\omega\times\eta\in \mathbb{P}\times\mathbb{P}$ unitarily equivalent modulo a rank one perturbations, that is, for almost all $(\omega, \eta)$ there exists a unitary operator $U$ and a vector $\f\in\cH$ such that $UH_\omega U^{-1} = H_\eta + (\fdot, \f)\f$. The proof relies on refining a method by Poltoratski \cite{alex2000, alex1998} which involves the Krein--Lifshitz spectral shift function. 

{\bf Vague direction:} Results of the above types are related to the Anderson localization conjecture. But even if one does not prove the full conjecture, it would be of interest to find out more about the (almost sure) spectral properties of Anderson-type Hamiltonians via studying rank one and finite rank perturbations; or vice versa.

%%%%%%%%%%%%%%%%%%%%%%%%%%%%%%%%%%%%%%%%%%%%%%%%
\section{Sz.-Nagy--Foia\c s model theory and rank $n$ unitary perturbations}\label{s-GenMod}
%%%%%%%%%%%%%%%%%%%%%%%%%%%%%%%%%%%%%%%%%%%%%%%%
Let us introduce the relationship between the spectral theory of finite rank perturbations and Sz.-Nagy--Foia\c s model theory. We closely follow \cite{RonRkN} and the references therein. Much of this section can be generalized to infinite rank perturbations, if we assume defect operators of trace class.

The main objective is to study the spectral properties of the finite rank perturbations given by \eqref{d-UA} below via the function theoretic properties of the corresponding characteristic operator function $\Theta$.

Consider Hilbert spaces $\mathcal{E}$ and $\mathcal{E}_\ast$ with $\dim(\mathcal{E}) = \dim(\mathcal{E}_\ast)=n<\infty$. Let $\Theta\in H^\infty(\mathcal{E}\to \mathcal{E}_\ast)$ with operator norm $\|\Theta(z)\|\ci{\mathcal{E} \to \mathcal{E}_\ast}\le 1$ for all $z\in\D$. Without loss of generality, we assume $\Theta(0)$ is the zero operator. Then the non-tangential boundary limits of $\Theta$ exist in the operator norm topology a.e.~Lebesgue $m$. We denote this limit also by $\Theta$.
An operator-valued function $\Theta$ is said to be \emph{inner}, if $\Theta(\xi)$ is a unitary operator for $\xi\in\T$ a.e.~$m$.

Let $\OID_\mathcal{E}$ denote the identity operator on $\mathcal{E}$.
Define $\bigtriangleup = (\OID_\mathcal{E} - \Theta \Theta^\ast)^{1/2}$.
The Sz.-Nagy--Foia\c s model space is given by
\begin{align}\label{d-Ktea}
\mathcal{K}_\Theta=\Kft.
\end{align}

Let $P_\Theta$ denote the orthogonal projection from $\begin{pmatrix}H^2(\mathcal{E}_\ast)\\ \clos\bigtriangleup L^2(\mathcal{E})\end{pmatrix}$ onto $\mathcal{K}_\Theta$. Further, consider the operators $T_z$ and $M_\xi^{\bigtriangleup}$ acting as multiplication by the independent variables $z\in\D$ and $\xi\in\T$, respectively.
A \emph{completely nonunitary} (cnu) contraction is an operator $X$ that satisfies $\|X\|\le 1$ and is not unitary on any of its invariant subspaces.
One can show that the following contraction is cnu:
\begin{align}\label{d-cnu}
T_\Theta = P_\Theta \begin{pmatrix} T_z & \OZ\\ \OZ & M_\xi^{\bigtriangleup} \end{pmatrix}
\qquad\text{on }\mathcal{K}_\Theta.
\end{align}

Let us explain briefly how Sz.-Nagy--Foia\c s model theory allows the study of all cnu contractions on a separable Hilbert space. Consider a contraction $U_\OZ $ on a separable Hilbert space $\cH$ with deficiency indices $(n,n)$ and defect spaces $\cD$ and $\cD_\ast$.

It is well-known that for any given cnu contraction $U_\OZ $ with deficiency indices $(n,n)$ on a separable Hilbert space $\cH$, there exists an operator-valued function $\Theta\in H^\infty (\cD\to\cD_\ast)$ such that the operators $(U_\OZ  \text{ on }\cH)$ and $(T_\Theta \text{ on } \mathcal{K}_\Theta)$ are unitarily equivalent (see for example, page 253 of \cite{SzNF2010}).

The function $\Theta$ is called the \emph{characteristic operator function} of the cnu contractions $U_\OZ $ and $T_\Theta$.

Let us outline the relationship of finite rank perturbations with model theory. 
With a rank $n$ unitary operator $A:\cD\to \cD_\ast$, Fuhrmann \cite{fuhrmann} constructed a unitary operator $U_A$ on $\cH$ by setting
\begin{align}\label{d-UA}
\begin{array}{rcccl}
\quad U_A:&\cH\ominus \cD&\longrightarrow&\cH\ominus \cD_\ast&
\quad\qquad\text{with }U_A f=U_\OZ f\text{ for } f\in\cH\ominus \cD\text{ and} \\
&\oplus&&\oplus\\
 U_A:&\cD&\longrightarrow&\cD_\ast &
\quad\qquad\text{by }U_A f=A f\text{ for } f\in\cD.
\end{array}
\end{align}

Similarly, we can view $U_A$ as rank $n$ perturbations of $U_\OZ$ on $\cH$ as follows.
Define the operators $P$ and $P_\ast$ to be the orthogonal projections of $\cH$ onto $\cD$ and $\cD_\ast$,
respectively. Further, we use the notation $P^\perp = I-P$ and $P^\perp_\ast = I-P_\ast$.
Consider the family of perturbations of $U_\OZ$ given by
\begin{align*}
 U_A = P^\perp_\ast U_\OZ P^\perp \oplus A P
 \qquad A: \cD\to \cD_\ast\text{ with unitary }A.
\end{align*}
It is not hard to see that $\cD_\ast$ is a cyclic subspace for $U_A$ for each unitary $A$.

Using that $\Theta(z) =  -U_\OZ  +z D_\ast (\OID - z U_\OZ ^*)^{-1} D$ with the defect operators $D$ and $D_\ast$, a formula for the characteristic operator function $\Theta$ was first obtained in \cite{BallLubin, fuhrmann}. Although, the operator function $S(z)$ that occurs on the right hand side of their formula is related to the operator-valued characteristic function $\Theta_A(z)$ via a unitary change of basis, the formula was used to prove some interesting spectral properties of $U_A$ and the higher rank unitary perturbations in terms of $\Theta$.
A more explicit formula for $\Theta$ was derived in Theorem 4.1 of \cite{RonRkN}.

Let $\mu_A$ denote the scalar-valued spectral measure; that is, we have the direct  integral
\begin{align}\label{d-directsum}
\cH = \oplus \int_\T \cH(\xi) d\mu_A(\xi)
\qquad\text{with }
U_A=\oplus \int_\T \xi\, \OID_{\cH(\xi)} d\mu_A(\xi),
\end{align}
where $ \OID_{\cH(\xi)} $ denotes the identity operator on the Hilbert space $ \cH(\xi) $.
Since $\cD_\ast$ is a cyclic subspace for $U_A$ of minimal dimension, we have $\max \dim\cH(\xi) = n$.

We re-write this perturbation as a problem on $
\mathcal{K}_\Theta = 
\begin{pmatrix}H^2(\cD_\ast)\vspace{.4mm}\\\clos\bigtriangleup L^2(\cD)\end{pmatrix}
\ominus \begin{pmatrix}\Theta\\\bigtriangleup\end{pmatrix} H^2(\cD),
$ where $\Theta$ is the characteristic operator function that corresponds to $U_\OZ$.

The following four statements are equivalent:
\begin{itemize}
\item[1)] The operators $(U_\OZ^\ast)^n$ converge to the zero operator in the strong operator topology as $n\to \infty$.
\item[2)] The model space reduces to $\mathcal{K}_\Theta=H^2(\cD)\ominus\Theta H^2(\cD)$; that is, the second component collapses to $\{0\}$.
\item[3)] The characteristic operator function $\Theta$ is inner.
\item[4)] The finite rank unitary perturbations $U_A$ all have purely singular spectrum, i.e.~$(\mu_A)\ti{ac}\equiv 0$ for all unitary $A$.
\end{itemize}

The equivalency of statements 1) and 3) can be found in \cite{BallLubin}, see also Proposition 3.5 of \cite{SzNF2010}. The fact that 2) $\Leftrightarrow$ 3) is trivial by virtue of the definition \eqref{d-Ktea} of $\mathcal{K}_\Theta$. The equivalency of parts 3) and 4) for general $n\in \N$ is the main result of \cite{RonRkN}. In the case $n=1$, the latter equivalence is implied by equation \eqref{e-hgamma} below. However, a generalization of the ``classical'' proof of this equivalence from $n=1$ to $n>1$ does not seem to be straightforward. {\bf Open problem:} It would be useful to extend the original proof to higher rank perturbations, since such an extension could yield the explicit density functions for the absolutely continuous part in the case of finite rank perturbations.

The equivalency of the above four statements explains why the case of purely singular spectrum is much simpler.

It is worth mentioning that the inner-outer factorization of operator-valued inner functions has proved to be involved. Concerning finite rank perturbations, this may mean that the situation for $n>1$ is significantly more difficult than $n=1$.

{\bf Open direction:} In the case of higher rank perturbations, we certainly cannot hope that the singular parts of two perturbed operators' spectral measures are mutually singular. However, it would be interesting to see how much of Aronszajn--Donoghue theory can still be true with appropriate changes in the statements. For example, instead of the scalar spectral measure, one would need to consider vector-valued measures.

%%%%%%%%%%%%%%%%%%%%%%%%%%%%%%%%%%%%%%%%%%%%%%%%%%%
\section{Aleksandrov--Clark Theory: Rank one unitary perturbations}\label{s-ACmeas}
%%%%%%%%%%%%%%%%%%%%%%%%%%%%%%%%%%%%%%%%%%%%%%%%%%%
We generalize Clark's setup to the case of non-inner characteristic functions which is often referred to as Aleksandrov--Clark theory. Rather than repeating the very accessible expositions \cite{cimaross, poltsara2006} of the substantial progress made by many mathematicians, we choose a slightly different, equivalent approach (see \cite{RonRkN} for further details) in order to highlight some of the differences between Clark and Aleksandrov--Clark theory.

When studying the relationship between the above setup for $n=1$ and the normalized Cauchy transform (below), we will mention another reason for this simplification.

Take $\te\in H^\infty(\D)$ with $\|\te\|\ci{H^\infty}\le 1$ and $\te(0)=0$.

Recall the definition of the Sz.-Nagy--Foia\c s model space $\mathcal{K}_\te$ and of the cnu contraction $T_\te$.
All rank one unitary perturbations of $T_\te$ are given by
	\[
	\widetilde U_\gamma=T_\te + \gamma \left(\fdot, \begin{pmatrix} z^{-1}\te\\ z^{-1}\bigtriangleup\end{pmatrix}\right)_{\mathcal{K}_\te}\begin{pmatrix}\ID_z\\0\end{pmatrix}
	\qquad\text{on }\mathcal{K}_\te\text{ for }|\gamma|=1,
	\]
where $ \left(\fdot, \begin{pmatrix} z^{-1}\te\\z^{-1}\bigtriangleup\end{pmatrix}\right)_{\mathcal{K}_\te}\begin{pmatrix}\ID_z\\0\end{pmatrix}$ denotes the rank one operator defined by
\[
\begin{pmatrix}f\\g\end{pmatrix} \mapsto \begin{pmatrix}\left[(f,z^{-1} \te)\ci{H^2} + (g,z^{-1}\bigtriangleup)\ci{L^2}\right]\ID_z\\0\end{pmatrix}.
\]
The defect spaces of $T_\te$ are $ \left<\begin{pmatrix}z^{-1}\te\\z^{-1}\bigtriangleup\end{pmatrix}\right>$ and $\left<\begin{pmatrix}\ID_z\\0\end{pmatrix}\right>$.

Let $\mu_\gamma$ denote the spectral measure of $\widetilde U_\gamma$ with respect to the cyclic vector $\kf{\ID_z}{0}$.

The Clark operator $\Phi_\gamma: \mathcal K_\te\to L^2(\mu_\gamma)$ for $\gamma\in \T$ intertwines $T_\theta$ and $U_0$ and it maps the defect spaces $\cD\ci{T_\theta}, \cD\ci{T_\theta^\ast}\subset \mathcal K_\te$ of $T_\theta$ to those of $U_0$. Here $U_0 = M_\xi - (\fdot, \bar \xi)\ID$ on $L^2(\mu_\gamma)$ was defined in Section \ref{s-GenMod} (for a general Hilbert space).
The identification of the defect spaces is proved in Lemma \ref{l-VKernels} below.

There is a one-to-one correspondence between the spectral families that arise from rank one unitary perturbations and the corresponding $\te\in H^\infty (\D)$ with $\|\te\|\ci{H^\infty} \le 1$. For $z\in \C\backslash\T$ we have by a Herglotz argument that
\[
 \frac{\gamma + \te(z)}{\gamma - \te(z)} = \int\ci\T \frac{\xi+z}{\xi-z}\,d\mu_\gamma(\xi),
\]
see for example, \cite{cimaross}. It is not hard to see that $\mu_\gamma$ is a probability measure. In fact, the latter condition is equivalent to our assumption that $\theta(0)=0$.

For a complex-valued measure $\tau$, the \emph{Cauchy transform} $K$ of the measure $f\tau$ is given by
\begin{align}\label{d-Cauchy}
K(f\tau)(z)=\int\ci\T\frac{f(\xi)d\tau(\xi)}{1-\bar\xi z},\qquad
z\in\C\backslash \T, f\in L^1(|\tau|).
\end{align}

The \emph{normalized Cauchy transform}
\begin{equation}\label{d-noCT}
\calC_{f\tau}(z)= \frac{K(f\tau)}{K\tau}(z)\,,\qquad z\in\C\backslash \T, f\in L^1(|\tau|),
\end{equation}
is an analytic function on $\C\backslash \T$. It is one of the central objects in the theory of rank one perturbations.

Poltoratski's theorem says that for a complex Borel measure $\tau$ on $\T$ and any $f\in L^1(|\tau|)$, the non-tangential limit of $\calC_{f\tau}$ exists and is equal to $f(\xi)$ with respect to the singular part a.e.~$|\tau|\ti{s}$
(see, for example, Poltoratski \cite{NONTAN}, also see \cite{NPPOL}).

\begin{rem}
Given a function $f\in H^\infty(\D)$ with $\|f\|\ci{H^\infty} \le 1$ and $f(0)=0$, one can use Poltoratski's theorem to show the existence (and obtain the value) of the non-tangential boundary value at some $\zeta\in\T$. Indeed, this can be done for some $|\gamma|=1$ by finding a corresponding Aleksandrov--Clark measure $\mu_\gamma$ that has a point mass at $\zeta$. One can then proceed in a similar way using the singular continuous part of $\mu_\gamma$.
\end{rem}

The role of the normalized Cauchy transform central object behind several important results in Clark theory.
If $\te$ is inner, we have $\Phi_\gamma^\ast f= \calC_{f \mu_\gamma}$ for all $f\in L^2(\mu_\gamma)$. For inner $\theta$, the latter formula plus Poltoratski's theorem, is the key to several results in Clark theory.

If $\theta$ is not inner, then $\Phi_\gamma^\ast f\in \mathcal{K}_\theta$ implies that it is two-valued. The situation is much more difficult. The author, in collaboration with S.~Treil, proved the following new (until now unpublished) representation in the case $\gamma=1$. We included the proof in the appendix, Section \ref{s-appendix}.

Let $P_\te$ be the orthogonal projection of $\begin{pmatrix}H^2\\ \bigtriangleup L^2\end{pmatrix}$ onto $\mathcal{K}_\theta$.

\begin{theo}\label{t-repr}
Under the above assumptions, we have
\begin{align}\label{f-repr}
\Phi_1^\ast f(z)&= P_\te\left[f(z)+ \int\frac{\xi(f(\xi)-f(z))}{\xi-z}\, d\mu(\xi)\right]\kf{\ID_z-\te}{-\bigtriangleup}
\end{align}
for $f\in C^1(\T)$.
\end{theo}

\begin{rems}
(a) Let us briefly mention that formula \eqref{f-repr} contains a slight abuse of notation (for a more precise explanation confer the remark below in the proof of this result): We take scalar multiplication of the vector with the function in square brackets, by taking its non-tangential boundary values into the second component (the first component is a function on the unit disc while the second one is on the unit circle). Projection $P_\te$ is then applied to the resulting vector.\\
(b) Formula \eqref{f-repr} is not an easy object of study. Taking the projection of a vector valued function containing singular integrals. One can think of the formula as applying two singular integral operators consecutively. {\bf Open problem:} Find an interpretation for this formula. In particular, find a reasonable expression for $\Phi_1^\ast f$ that does not involve $P_\theta$. The importance of the normalized Cauchy transform in the case of inner functions indicates that such a simplified representation will be very useful.
\end{rems}

Consider the Lebesgue decomposition $d\mu_\gamma = d(\mu_\gamma)\ti{ac} + d(\mu_\gamma)\ti{s}$, where $d(\mu_\gamma)\ti{ac} = w_\gamma dm$, $w_\gamma \in L^1(m)$.
Using analytic methods, one can express the density of the absolutely continuous part of the spectral measure $\mu_\gamma$ is given by
\begin{align}\label{e-hgamma}
w_\gamma(\xi) = \frac{1-|\theta(\xi)|^2}{|\gamma - \theta(\xi)|^2}\qquad\text{for }\xi\in\T\text{ a.e.~}m,
\end{align}
see for example, \cite{cimaross}, Proposition 9.1.14.

In \cite{AbaLiawPolt} the authors proved a certain Aronszajn--Krein type formula which, when translated to the case of rank one unitary perturbations, states that the normalized Cauchy transform of $(V_\gamma f)\mu_\gamma$, $ f \in L^2(\mu)$, is independent of $\gamma$ where $V_\gamma$ is the unitary operator $V_\gamma:L^2(\mu)\to L^2(\mu_\gamma)$ such that $V_\gamma U_\gamma =  M_\zeta V_\gamma$ and $V_\gamma \ID_\xi= \ID_\zeta$ and
the operator $M_\zeta$ acts as multiplication by the independent variable in $L^2(\mu_\gamma)$. More precisely, we have $\mathcal{C}_{(V_\gamma f)\mu_\gamma}=\mathcal{C}_{f\mu}$ for $f\in L^2(\mu)$.

In \cite{RonRkN} the authors obtained a very simple representation of the characteristic function of $U_0$, namely, $\theta(z) = z \calC_{\bar\xi\mu}(z)= \gamma z \calC_{\bar\xi\mu_\gamma}(z)$ for $z\in\D$, $|\gamma|=1$.

For $f\in L^2(\mu_\gamma)$, the normalized Cauchy transform $\calC_{f \mu_\gamma}$ is analytic in the open unit disc $\D$. It follows from the latter representation of the characteristic function that, if we take $f(\xi) = \bar\xi$, then $\calC_{\bar\xi\mu_\gamma}$ is also uniformly bounded on $\D$.

Further, the formula for $\te$ was used to explicitly express the non-tangential jump behavior of the normalized Cauchy transform of the measure $\bar\xi \mu$ across the unit circle. Let us briefly explain this result.
It follows from Poltoratski's theorem that the non-tangential boundary values of the normalized Cauchy transform from the inside and outside, respectively, coincide a.e.~$(\mu_\gamma)\ti{s}$. However, Poltoratski's theorem does not provide any information about the boundary values on the absolutely continuous part $(\mu_\gamma)\ti{ac}$. On that part of $\T$, the difference of the non-tangential boundary values of the normalized Cauchy transform (from the inside and outside, respectively) is some non-zero function, since $\calC_{f\mu}$ is an analytic function on $\C\backslash\T$. In \cite{RonRkN}, the authors found a simple explicit expression for this jump for non-tangential boundary limits of $\calC_{\bar\xi\mu_\gamma}$.

{\bf Open problems:}
Is it possible to control the jump of the normalized Cauchy transform of $f\mu_\gamma$, $f\in L^2(\mu_\gamma)$? This may be useful for generalizing the connection to pseudocontinuation, see Section \ref{s-pseudo} below.
Further relationships between the boundary values of $\theta$ and the spectral behavior of $U_\gamma$ would be of interest.

%%%%%%%%%%%%%%%%%%%%%%%%%%%%%%%%%%%%%%%%%%%%%%%%%%%
\section{Generalization of pseudocontinuation}\label{s-pseudo}
%%%%%%%%%%%%%%%%%%%%%%%%%%%%%%%%%%%%%%%%%%%%%%%%%%%
In the case where $n=1$ and the characteristic function $\te$ is inner, a famous result by Douglas--Shapiro--Shields \cite{DSS} relates the existence of a pseudocontinuation of a function to the question whether the function is an element of the model space $K_\te$. Pseudocontinuable functions have interesting connections in operator and function theory, especially concerning the cyclicity of functions (see, for example, \cite{cimaback, NikolskiTreat}).

If we do not assume that the characteristic function is inner, then the corresponding Nikolski--Vasyunin coordinate free model space consists of pairs of functions, and so does the Sz.-Nagy--Foia\c s transcription \eqref{d-Ktea}. In this case, Kapustin \cite{Kapustin1, Kapustin2} provided a unitarily equivalent formulation of rank one unitary perturbations of a cnu contraction via a model space consisting of pairs of functions, the first of which is meromorphic inside the unit disc and the second of which is meromorphic outside the unit disc, and the boundary values of which satisfy a certain jump condition.

Although Kapustin's main goal was to construct a model operator that is unitarily equivalent to a given operator, his result can be thought of as a generalization of pseudocontinuation by allowing a certain controlled jump behavior at the boundary. It involves the existence of two functions from $H^2$ in terms of which the jump condition of the pair of meromorphic functions is formulated.
His approach relies heavily on the coordinate free model space by Nikolski and Vasyunin.

{\bf Many open questions:} What do the different choices of $H^2$ functions have in common? Can we express these conditions in terms of the spectral properties of the class of rank one unitary perturbations of the related non-unitary operator? Is there a canonical way of choosing the $H^2$ functions?
What does Kapustin's result mean in some transcription of the coordinate free model theory? Does Kapustin's approach extend to rank $n$ unitary perturbations where $n>1$ (certain generalizations of the boundary convergence of pseudocontinuable functions were proved to the vector-valued case, see \cite{AlexKapustin})?
Finally, how are the two meromorphic functions connected to the well-studied objects from Aleksandrov--Clark theory, for example, the normalized Cauchy transform. More precisely, can we relate the pair of functions to the non-tangential jump behavior of the normalized Cauchy transform $\mathcal{C}_{\bar\xi\mu}$?

%%%%%%%%%%%%%%%%%%%%%%%%%%%%%%%%%%%%%%%%%%%%%%%%%%%
\section{Certain other models}\label{s-last}
%%%%%%%%%%%%%%%%%%%%%%%%%%%%%%%%%%%%%%%%%%%%%%%%%%%
Apart from Sz.-Nagy--Foia\c s theory, there are several other model theoretic approaches to finite rank perturbations, for example, via the de Branges--Rovnyak space (see \cite{DB1, DB2}, also see \cite{SAR}). Recently, Kapustin \cite{Kapustin2} introduced another interesting approach to rank one perturbations.
Let us explain the larger framework between the latter three approaches in the case of rank one perturbations.

Consider a cnu contraction $U_0$ on a separable Hilbert space with deficiency indices $(1,1)$. Via model theory, such an operator has a characteristic function $\theta$ which is an analytic self-map of the unit disk. The idea is to study the operator theoretic properties of the rank one unitary perturbations of $U_0$ via the function theoretic properties of $\theta$. Consider $\bigtriangleup(\xi) = \sqrt{1-|\theta(\xi)|^2}$ for $\xi\in\T$ (which is defined Lebesgue a.e.).

If $\theta$ is inner (i.e.~$|\theta(z)|=1$ for $z\in\T$ Lebesgue a.e.), then $\bigtriangleup \equiv 0$ and $\int \log \bigtriangleup = -\infty$. This case is studied in Clark theory. In Clark theory the map from the model space to the Hardy space $H^2$ is one-to-one.
Aleskandrov--Clark theory uses many similar methods but does not assume the characteristic function to be inner. As a result the map from the model space to the Hardy space $H^2$ is not one-to-one and many complications occur. The full (two-valued) version of the Sz.-Nagy--Foia\c s model space must be studied.

The approaches of de Branges--Rovnyak  and Kapustin also do not assume that $\theta$ is inner. Kapustin's restriction is equivalent to the condition $\int\log|\theta|>-\infty$. {\bf Open question:} Does this condition have an interpretation in terms of the spectrum?

When using de Branges--Rovnyak spaces, there are two cases. If $\int \log \bigtriangleup = -\infty$, then the model simplifies, as the map from the model space to $\bigtriangleup L^2$ is one-to-one. This case happens if and only if $\theta$ is extreme. The latter is moreover equivalent to the polynomials being dense in the set of square integrable functions with respect to the spectral measure all of the perturbed operators. However, if $\int \log \bigtriangleup > -\infty$, then the corresponding map is not one-to-one and the situation becomes more complicated.
It is worth mentioning that in this case we have $\bigtriangleup>0$ and $|\theta|<1$ Lebesgue a.e., and by equation \eqref{e-hgamma}, it follows that $\sigma\ti{ac} = \T$.

%%%%%%%%%%%%%%%%%%%%%%%%%%%%%%%%%%%%%%%%%%%%%%%%%%%
\section{Appendix: Proof of Theorem \ref{t-repr}}\label{s-appendix}
%%%%%%%%%%%%%%%%%%%%%%%%%%%%%%%%%%%%%%%%%%%%%%%%%%%
Let operators $\widetilde U_1$, $T_\te$, $K_\te$, $\Phi^\ast_1$ and the spectral measure $\mu$ be defined as in Section \ref{s-ACmeas}. In order to simplify notation let $V=\Phi^\ast_1$. Recall that $V:L^2(\mu) \to \mathcal{K}_\theta$.

It suffices to prove two separate formulas for monomials in $\xi$ and in $\bar\xi$, namely
\begin{align}
V\xi^n
&=
P_\te\left[z^n+ \int\frac{\xi(\xi^n-z^n)}{\xi-z}\, d\mu(\xi)\right]x(z),
\qquad n\in\N\cup\{0\},\label{f-posn}\\
V\bar\xi^n
&=
P_\te\left[\bar z^n x(z)- \int\frac{\xi(\bar\xi^n-\bar z^n)}{\xi-z}\, d\mu(\xi)z^{-1}y(z)\right],
\qquad n\in\N,\label{f-negn}
\end{align}
where
\begin{align}\label{f-xy}
x(z)=V \ID_\xi= \kf{\ID_z}{0}\in K_\te,
\qquad
y(z)=V\bar \xi=\kf{ z^{-1} \te}{ z^{-1} \bigtriangleup} \in K_\te.
\end{align}

Indeed, to see the sufficiency of \eqref{f-posn} and \eqref{f-negn} note that $\kr \in\kr H^2$ and $ z^{-1}\in H_-^2$. So, in $K_\te$, we have
\[
x(z)=\kf{\ID_z-\te}{-\bigtriangleup}
\qquad\text{and}\qquad
y(z)=\kf{ z^{-1}\te- z^{-1}}{ z^{-1}\bigtriangleup}\,,
\]
which implies that $x(z)=-zy(z)$ in $K_\te$.
The representation theorem \ref{t-repr} follows for monomials $f(\xi)=\xi^k$, $k\in \Z$, and by linearity for polynomials.

Let us extend this formula to $C^1(\T)$. Take arbitrary $f\in C^1(\T)$.  Take a sequence of polynomials $\{p_k\}$ such that $p_k\to f$,  $p_k'\to f'$ uniformly on $\T$. In particular we have $p_k\to f$ in $L^2(\mu)$ and $|p_k'|\le C$ (with constant $C$ independent of $k$).  
We substitute the polynomials  $p_k$ into \eqref{f-repr} and take the limit as $k\to \infty$.  Recall that $V:L^2(\mu)\to K_\te$ is a unitary operator. So on the left hand side we get $Vp_k\to Vf$ in $K_\te$. By the mean value theorem  and Lebesgue dominated convergence we have
\[
\int\frac{\xi(p_k(\xi)-p_k(z))}{\xi-z}\, d\mu(\xi)
\,\,\, \longrightarrow \,\,\,
\int\frac{\xi(f(\xi)-f(z))}{\xi-z}\, d\mu(\xi)
\qquad\text{for all }z\in \T.
\]

Taking a subsequence $p_{k_n}$ such that $p_{k_n} \to f$ Lebesgue a.e., and $V p_{k_n}\to Vf$ Lebesgue a.e. (in both components) we get Theorem \ref{t-repr}.

Therefore, it remains to prove equations \eqref{f-posn}, \eqref{f-negn} and \eqref{f-xy}.

Let us begin by proving \eqref{f-posn} (minus the identities \eqref{f-xy} which are proven below in Lemma \ref{l-VKernels}).
Recall that $V:L^2(\mu)\to K_\te$ such that $T_\te V= V U_0$ is the adjoint of the Clark operator. Taking into account that $ U_0 = M_\xi - (\fdot, \bar\xi)\ID$, we obtain
\begin{align}\label{f-inter}
T_\te V= V \bigl[ M_\xi -  (\fdot, \bar \xi )\ci{L^2(\mu)}\ID_z\bigr].
\end{align}
Let us proceed by induction, using
\begin{align}\label{f-ind1}
VM_\xi &= T_\te V +  (\fdot, \bar \xi )\ci{L^2(\mu)} x(z), \qquad x(z)=V \ID_\xi
\end{align}
for $n=1$ and making the assumption
\begin{align}\label{f-indass}
VM_\xi^n &= T_\te^n V +  \sum_{k=0}^{n-1} (\fdot, \bar \xi^{k+1} )\ci{L^2(\mu)} T_\te^{n-k-1}x(z).
\end{align}
We finish the induction with the computation
\begin{align*}
VM_\xi^{n+1} 
&= (VM_\xi) M_\xi^{n}=T_\te (VM_\xi^{n}) +  (M_\xi^{n}\fdot, \bar \xi )\ci{L^2(\mu)} x(z)\\
&=
T_\te^{n+1} V +  \sum_{k=0}^{n-1} (\fdot, \bar \xi^{k+1} )\ci{L^2(\mu)} T_\te^{n-k}x(z)+ (\fdot, \bar \xi^{n+1} )\ci{L^2(\mu)} x(z)\\
&=
T_\te^{n+1} V +  \sum_{k=0}^{n} (\fdot, \bar \xi^{k+1} )\ci{L^2(\mu)} T_\te^{n-k}x(z)
\end{align*}
where we applied \eqref{f-ind1} and \eqref{f-indass}.

It is not hard to see that $T_\te^n=P_\te \begin{pmatrix} T_z^n &0\\0& (M_{e^{it}}^\bigtriangleup)^n\end{pmatrix}$.

\begin{rem}
Since this matrix is diagonal, it essentially behaves like a scalar. Slightly abusing notation, and WLOG, we write $T_\te^n\cdot_z=P_\te M_z^n\cdot_z = P_\te z^n\cdot_z$. This explains the notation in Theorem \ref{t-repr}.
\end{rem}

If we can show that
\begin{align}\label{f-geo}
\sum_{k=0}^{n-1} \xi^{k+1} P_\te z^{n-k-1}x(z)
&=
P_\te \,\frac{\xi(\xi^n-z^n)}{\xi-z}\,x(z),
\end{align}
then \eqref{f-posn} is proven.

Let us prove formula \eqref{f-geo}.
By telescoping sums we obtain
\begin{align*}
&P_\te (\xi^n-z^n)\,\cdot_z=(\xi^n-P_\te z^n)\,\cdot_z=(\OID-P_\te z\bar\xi)\,\sum_{k=0}^{n-1}   \xi^{k+1} P_\te z^{n-k-1}\,\cdot_z\\
&=P_\te(1- z\bar\xi)\,\sum_{k=0}^{n-1}   \xi^{k+1} P_\te z^{n-k-1}\,\cdot_z
=P_\te(1- z\bar\xi)\,\sum_{k=0}^{n-1}   \xi^{k+1}  z^{n-k-1}\,\cdot_z\\
&=\sum_{k=0}^{n-1}   \xi^{k+1} P_\te z^{n-k-1}(1- z\bar\xi)\,\cdot_z
=\sum_{k=0}^{n-1}   \xi^{k+1} P_\te z^{n-k-1}\bar\xi(\xi- z)  \,\cdot_z,
\end{align*}
and identity \eqref{f-geo} follows immediately.

Formula \eqref{f-negn} (minus \eqref{f-xy}) is proven in analogy by using the adjoint of $T_\te V= V U_0$.
%, we obtain the equation
%\[
%VM_{\bar\xi} = T_\te^* V +  (\fdot, \ID_\xi )\ci{L^2(\mu)} y(z), \qquad y(z)=V \bar\xi,
%\]
%which corresponds to equation \eqref{f-ind1}. After induction we have
%\[
%V(M_{\bar\xi})^{n} 
%=
%(T_\te^*)^{n} V +  \sum_{k=0}^{n-1} (\fdot, \xi^{k} )\ci{L^2(\mu)} (T_\te^*)^{n-k-1}y(z),
%\]
%in analogy to equation \eqref{f-indass}.
%Finally, equation \eqref{f-geo} is replaced by
%\[
%\sum_{k=0}^{n-1} \bar\xi^{k} P_\te \bar z^{n-k-1}y(z)
%=
%P_\te \,\frac{\xi(\bar\xi^n-\bar z^n)}{\xi-z}\,zy(z).
%\]
%This proves equation \eqref{f-negn}.

In the proof of Theorem \ref{t-repr}, it remains to show \eqref{f-xy}.

\begin{lem}\label{l-VKernels}
%The defect operators of $T_\te$ are
%\begin{align}\label{f-defS}
%I-T_\te T_\te^*= \kf{\ID}{0}\kf{\ID}{0}^*
%\qquad\text{and}\qquad
%I-T_\te^* T_\te= \kf{\bar z \te}{\bar z \bigtriangleup}\kf{\bar z \te}{\bar z \bigtriangleup}^*\,.
%\end{align}
The defects of $T_\theta$ are given by $V\bar\xi=\kf{z^{-1} \te}{z^{-1} \bigtriangleup}$ and $V\ID_\xi=\kf{\ID_z}{0}$.
\end{lem}

\begin{proof}[Proof of Lemma \ref{l-VKernels}]
Let us use the notation $ b   b ^*= b  ( \fdot,  b )$ for $b\in L^2(\mu)$. Observe that
\begin{align}\label{f-defU}
I-U_0U_0^*= \ID_\xi \ID_\xi^* ,
\qquad
I-U_0^* U_0= \bar\xi (\bar\xi)^*
\qquad\text{on }L^2(\mu).
\end{align}
Indeed, by definition of $U_0$ we obtain
\[
 U_0= M_\xi - \ID_\xi(\bar \xi)^*= (I-\ID_\xi\ID_\xi^*)M_\xi
\qquad\text{and}
\qquad
U_0^*= M_{\bar\xi}(I-\ID_\xi\ID_\xi^*).
\]

Let us show that $T_\te V\bar\xi=\vec{0}$. Recall that $V:L^2(\mu)\to K_\te$ is a unitary operator such that $T_\te V= V U_0$. Since we assumed that $\|\bar\xi\|\ci{L^2(\mu)}=\mu(\T)=1$, we have $%\|V\ID\|\ci{K_\te}=
\|V\bar\xi\|\ci{K_\te}=1$. So the second identity of \eqref{f-defU} implies that
\begin{align*}
%(I-T_\te T_\te^*)V\ID= V\ID 
%\qquad\text{and }\qquad
(I-T_\te^* T_\te)V\bar\xi = V\bar\xi.
\end{align*}
Hence %$V\ID$ and 
$V\bar\xi$ is the kernel of %$T_\te T_\te^*$ and 
$T_\te^* T_\te$%, respectively
. Because $(T_\te^* T_\te V\bar\xi, V\bar\xi)=\|T_\te V\bar\xi\|^2$ we have $T_\te V\bar\xi=\vec 0$.

In analogy, it follows that $T_\te^*V\ID_\xi=\vec{0}$.

Note that the kernels of $T_\te$ and $T_\te^*$ are one dimensional, since the definciency indicees of $U_0$ were $(1,1)$.

Let us show that $V\ID_\xi=c_1\kf{\ID_z}{0}$, $c_1\in\R$. Take $\vec h=\kg\in K_\te$ such that $T_\te^* \vec h=\vec 0$. Notice that $T_\theta^\ast  =  z^{-1}\kf{f-f(0)}{g}$ where the second component is zero only if $ z^{-1} g\in \bigtriangleup H^2$. It follows that $g\in \bigtriangleup H^2$. On the other hand, since $\vec h \in K_\te$, we have $g\perp \bigtriangleup H^2$. So we must have $g\equiv 0$. In the first component, we obtain $f-f(0)\in \te H^2$ while $f\perp \te H^2$. From $(f-f(0),f)=0$, it easily follows that $f=f(0)$ is a constant function. The vector $\vec h=\kf{\ID_z}{0}$ is the only vector such that  $T_\te^*\vec h=\vec 0$ and must hence be in $K_\te$.
%In particular, it follows that $\ID_z\in H^2\ominus \te H^2$, i.e.~$\ID_z\perp \te H^2$. So we have $\te(0)=0$.

With the definition of $K_\te$ it is easy to see that $V\bar\xi=c_2\kf{z^{-1} \te}{z^{-1} \bigtriangleup}$, $c_2\in\R$. Indeed, if $T_\te\kg=\vec{0}$, then we must have $zf\in\te H^2$ while $f\perp \te H^2$, and similarly for the second component. Because $\te(0)=0$, it follows that vector $\kf{z^{-1} \te}{z^{-1} \bigtriangleup}\in K_\te$ satisfies these conditions.

It remains to observe that $\left\|\kf{\ID_z}{0}\right\|\ci{K_\te}=\left\|\kf{ z^{-1} \te}{ z^{-1} \bigtriangleup}\right\|\ci{K_\te}=1$, so that $c_1=c_2=1$.
\end{proof}

The representation theorem, Theorem \ref{t-repr}, is proved.\hfill\qed

\providecommand{\bysame}{\leavevmode\hbox to3em{\hrulefill}\thinspace}
\providecommand{\MR}{\relax\ifhmode\unskip\space\fi MR }
\providecommand{\MRhref}[2]{%
  \href{http://www.ams.org/mathscinet-getitem?mr=#1}{#2}
}
\providecommand{\href}[2]{#2}

\end{document}